\documentclass[a4paper,10pt]{article}

\title{A Null Ideal for Inaccessibles}

\author{Sy-David Friedman (KGRC, Vienna) and Giorgio Laguzzi (U.Freiburg)}

\usepackage{fontenc}
\usepackage{amsmath}
\usepackage{amsfonts}
\usepackage{amssymb}
\usepackage{amsthm}
\usepackage{graphicx}
\usepackage[english]{babel}

\newtheorem{definition2}{Definition}

\newtheorem{remark2}[definition2]{Remark}

\newtheorem*{Mobservation2}{Main Observation}

\newtheorem{example2}[definition2]{Example}

\newtheorem{proposition}[definition2]{Proposition}

\newenvironment{definition}{\begin{definition2} \upshape}{\end{definition2}}

\newenvironment{remark}{\begin{remark2} \upshape}{\end{remark2}}

\newcommand{\cantor}{2^\omega}

\newcommand{\cohen}{\poset{C}}

\newcommand{\conc}{\smallfrown}

\newcommand{\cov}{\textsf{cov}}

\newcommand{\club}{\textsf{Cub}}

\newcommand{\clubsacks}{\sacks^\club}

\newcommand{\DDelta}{\mathbf{\Delta}}

\newcommand{\enfa}{\textit}

\newcommand{\ideal}{\mathcal}

\newcommand{\IF}{\ideal{I}_{\T}}

\newcommand{\force}{\Vdash}

\newcommand{\height}{\mathsf{ht}}

\newcommand{\meager}{\ideal{M}}

\newcommand{\ns}{\textsf{NS}}

\newcommand{\poset}{\mathbb}

\newcommand{\rank}{\textsf{Rank}}

\newcommand{\restric}{{\upharpoonright}}

\newcommand{\sacks}{\poset{S}}

\newcommand{\SSigma}{\mathbf{\Sigma}}

\newcommand{\splitting}{\textsf{Split}}

\newcommand{\stem}{\textsf{Stem}}

\newcommand{\successor}{\textsf{succ}}

\newcommand{\T}{\poset{F}}

\newcommand{\term}{\textsc{Term}}

\begin{document}

\maketitle

\begin{abstract}

In this paper we introduce a tree-like forcing notion extending some properties of the random forcing in the context of
$2^\kappa$, $\kappa$ inaccessible, and study its associated ideal of null sets and notion of measurability.
This issue was also addressed by Shelah (\cite[Problem 0.5]{Sh14}) and concerns the definition of a forcing which is $\kappa^\kappa$-bounding, $<\kappa$-closed and $\kappa^+$-cc, for $\kappa$ inaccessible. This also contributes to a line of research adressed in the survey paper \cite{KLLS16}.

\end{abstract}

\section{Introduction}

In \cite[Problem 0.5(2)]{Sh14}, Shelah addresses the following question:

\vspace{2mm}

Can one define a forcing which is $\kappa^+$-cc, $<\kappa$-closed and $\kappa^\kappa$-bounding, for $\kappa$ inaccessible?

\vspace{2mm}

In \cite{Sh12}, Shelah provides a positive answer when $\kappa$ is weakly compact. More recently, Cohen and Shelah (see \cite{SC16}) introduce a forcing notion satisfying the three properties mentioned for $\kappa$ inaccessible, and not necessarily weakly compacy. A construction different from the one described by Cohen and Shelah \cite{SC16} is presented here. For this purpose, we need a certain version of $\Diamond$, modelling our construction on work of Jensen \cite{J68} in the case $\kappa=\omega$. 
The main differences from \cite{SC16} are the following: we use different versions of diamond; the use of diamond is different; and finally, our construction is somehow simpler, but also it is more restrictive, as our diamond hypothesis implies $2^\kappa=\kappa^+$. \footnote{We thank the referee to make us aware of \cite{SC16}, and to point these differences out.}

Note that in the standard case, i.e., when $\kappa=\omega$, the random forcing fits with the three properties. So the first attempts to resolve such a problem might be to generalize it for $\kappa$ inaccessible. Hence, since random forcing is usually defined by means of the Lebesgue measure in $\cantor$, the natural way would be to define an appropriate measure in $2^\kappa$ as well. Nevertheless, there seem to be many obstacles in trying to do so. Our method for defining our forcing $\T$ will not use any notion of measure.

Note that even if we will not define a measure on $2^\kappa$, we will define an ideal of $\T$-null sets and a notion of measurability associated to it, in a rather standard way. We will then investigate such a regularity property.

The paper is organized as follows. In section \ref{main} we present the construction of our forcing $\T$. Section \ref{ideal} is devoted to introducing the ideal of null sets and the notion of measurability, and to proving some (negative) results about $\DDelta^1_1$ sets and the club flter. A final section is then dedicated to some concluding remarks and possible further developments.

The first author wishes to thank the FWF (Austrian Science Fund) for its support through 
Project P25748.

\section{Preliminaries}

In this preliminary section we simply introduce the basic notions and notation which is needed throughout the paper.

\begin{itemize}

\item A tree $T$ is a subset of $2^{<\kappa}$, closed under initial segments. $\stem(T)$ denotes the longest node of $T$ compatible with all the other nodes of $T$; $\successor(t,T):=\{\xi <\kappa: t^\conc \xi \in T \}$; \splitting(T) is the set of splitting nodes of $T$ (i.e., $t \in \splitting(T)$ if both $t^\conc 0$ and $t^\conc 1 \in T$); we put $\height(T):= \sup \{\alpha: \exists t \in T (|t|=\alpha) \}$, while $\term(T)$ denotes the \enfa{terminal} nodes of $T$ (i.e., $t \in \term(T)$ if there is no $t' \supsetneq t$ such that $t' \in T$). For $\alpha < \kappa$,  $T\restric \alpha:= \{ t \in T: |t| < \alpha \}$. A \emph{branch} through $T$ of height $\kappa$ is the limit of an increasing cofinal sequence $\{ t_\xi: \xi < \kappa \}$ of nodes in $T$, and $[T]$ will denote the set of all branches of $T$. For $t$ in a tree $T$, $T_t$ is the set of nodes in $T$ compatible with $t$. 

\item Given $t \in \splitting(T)$, we define the rank of $t$ as the order type of $\{\alpha <|t|: t \restric \alpha \in \splitting(T)\}$. Furthermore, we let $\splitting_\beta(T)$ denote the set of splitting nodes in $T$ of rank $\beta$. The forcing $\clubsacks$ consists of trees $T$ such that every node can be extended to a splitting node and for some club $C$ in $\kappa$, the splitting nodes of $T$ are exactly those with length in $C$. For $T,T'$ in $\clubsacks$ and $\gamma<\kappa$, we write $T\leq_\gamma T'$ iff $T$ is a subtree of $T'$ such that $\splitting_\gamma(T)=\splitting_\gamma(T')$; $\leq_0 $ is simply denoted by $\leq$. 

Note that $\clubsacks$ is closed under $\leq_\gamma$-descending sequences of length less than $\kappa$ for each $\gamma<\kappa$. 

\item If $\{ \T_\alpha: \alpha < \kappa \}$ is a sequence of families of trees such that $\T_\alpha \subseteq \T_{\alpha+1}$, then for every tree $T \in \bigcup_{\alpha<\kappa} \T_\alpha$ define $\rank(T)$ to be the least $\alpha<\kappa$ such that $T \in \T_{\alpha+1}$. $\T_{<\lambda}$ denotes the union of the $\T_\alpha$ for $\alpha<\lambda$. 


\item A forcing $\poset{P}$ is called $\kappa^\kappa$-bounding iff for every $x \in \kappa^\kappa \cap V^{\poset{P}}$ there exists $z \in \kappa^\kappa \cap V$ such that $\force \forall \alpha < \kappa (x(\alpha) < z(\alpha))$.

\item In this paper $\cohen$ refers to the $\kappa$-Cohen forcing, i.e., the poset consisting of $t \in 2^{<\kappa}$, ordered by extension. The elements of $2^\kappa$ and $\kappa^\kappa$ are called $\kappa$-reals. 

\end{itemize}

Under the assumption $2^{<\kappa}=\kappa$, $\cohen$ is obviously
$\kappa^+$-cc, but also adds unbounded $\kappa$-reals, which means it
is not $\kappa^\kappa$-bounding. For $\kappa$ inaccessible
$\clubsacks$ is $\kappa^\kappa$-bounding, but one loses
the $\kappa^+$-cc. The next section is devoted to defining a refinement of
$\clubsacks$ in order to obtain the $\kappa^+$-cc and maintain
$\kappa^\kappa$-boundedness.

\section{The main construction} \label{main}

Fix $\kappa$ to be inaccessible. As we mentioned before, our first main goal is to define a tree-forcing $\T$ with the following three properties: $\kappa^+$-cc, $\kappa^\kappa$-bounding and $<\kappa$-closure. Assume $\Diamond_{\kappa^+}(S_\kappa^{\kappa^+})$, where $S_\kappa^{\kappa^+}:= \{ \lambda < \kappa^+: \text{cf}(\lambda)=\kappa \}$.

We construct an increasing  sequence of tree forcings $\langle \T_\lambda: \lambda < \kappa^+ \rangle$ by induction on $\lambda < \kappa^+$. We first remark that for all $\lambda < \kappa^+$ we will maintain the following: 

\begin{itemize}

\item[(P1)] $\T_\lambda \subseteq \clubsacks$ and $|\T_\lambda| \leq \kappa$;

\item[(P2)] $ \forall T \in \T_{<\lambda} \forall \gamma < \kappa \exists T' \leq_\gamma T \forall T'' \leq T'(T' \in \T_\lambda \land T'' \notin \T_{<\lambda})$;

\item[(P3)] $\forall T \in \T_\lambda \forall t \in T (T_t \in \T_\lambda)$;

\item[(P4)] $\T_\lambda$ is closed under descending $<\kappa$-sequences;

\item[(P5)] $\forall \alpha < \lambda \forall T \in \T_\lambda
 \setminus \T_\alpha \exists \bar \gamma < \kappa \forall \gamma \geq
 \bar \gamma \forall t \in \splitting_\gamma(T) \exists S \in
\T_\alpha \setminus \T_{<\alpha} (T_t \subseteq S)$. 

\end{itemize}

We remark that in P2 the property we are really interested in is P2bis: $ \forall T \in \T_{<\lambda} \forall \gamma < \kappa \exists T' \leq_\gamma T (T' \in \T_\lambda \setminus \T_{<\lambda})$; the extra requirement on all $T'' \leq T'$ is only needed to make sure that such a property will be preserved in our recursive construction.

Furthermore, P5 will be used to help ensure the $\kappa^+$-cc.

\vspace{2mm}

Let $\{ D_\lambda: \lambda < \kappa^+ \}$ be a $\Diamond_{\kappa^+}(S^\kappa_{\kappa^+})$-sequence. The recursive construction is developed as follows:

\begin{enumerate}

\item $\T_0 :=\{ (2^{<\kappa})_t : t \in 2^{<\kappa} \}$.

\item Case $\lambda+1$:

For every $T \in \T_\lambda\setminus \T_{<\lambda}$ and $\gamma < \kappa$, pick $T' \in \clubsacks$ such that $T' \leq_\gamma T$ and $T'$ does not contain subtrees in $\T_{\lambda}$; this is possible as $\T_\lambda$ has cardinality $\kappa$. Then for all $t \in T'$ we add $T'_t$ to  $\T_{\lambda+1}$.

Then for every $T \in \T_{\lambda+1}\setminus \T_{\lambda}$, $S \in \T_{\lambda+1}$ and $\sigma: T \rightarrow 2^{<\kappa}$ the canonical isomorphism, put $\sigma^{-1}[S] \in \T_{\lambda+1}$. (This is indeed needed to have a suitable notion of $\T$-measurability as explained in Proposition \ref{meas-regular}.)

We finally close $\T_{\lambda+1}$ under descending $<\kappa$-sequences, i.e., for every descending $\{  T^i: i < \delta \}$ in $\T_{\lambda+1}$, with $\delta < \kappa$, we put $T^* := \bigcap_{i<\delta}T^i$ into $\T_{\lambda+1}$. 

\item Case $\text{cf}(\lambda)< \kappa$:  
Let $\T_\lambda$ be the closure of $\T_{<\lambda}$ under descending $<\kappa$-sequences.

\item Case $\text{cf}(\lambda)= \kappa$, where $(\lambda_i : i < \kappa)$ is increasing and cofinal in $\lambda$: 

\begin{enumerate}

\item Suppose $D_\lambda \subseteq \lambda$ codes a maximal antichain $A_\lambda$ in $\T_{<\lambda}$. For every $T \in \T_{<\lambda}$ and $\gamma < \kappa$, construct a ``$\kappa$-fusion'' sequence $\{ T^i: i < \kappa \}$ of trees in $\clubsacks$ such that

\begin{enumerate}

\item $T=:T^0 \geq_\gamma T^1 \geq_{\gamma+1} T^2 \geq_{\gamma+2} \dots \geq_{\gamma+i} T^{i+1} \geq_{\gamma+i+1} \dots$

\item $T^i_t$ belongs to $\T_{<\lambda}$ with $\rank(T^i_t)$ at least $\lambda_i$ for each $t$ in $\splitting_\gamma(T)$. 

\item $T^1 := \bigcup\{ S_t: t \in \splitting_\gamma(T) \}$, where each $S_t \leq T_t$ and $S_t$ hits $A_\lambda$, i.e., there exists $S^* \in A_\lambda$ such that $S_t \leq S^*$.

\end{enumerate}

Then add $T^*:=\bigcap_{i <\kappa} T^i$ to $\T_\lambda$. Moreover, for every $t \in T^*$, add $T^*_t$ to $\T_\lambda$ too. Finally close $\T_\lambda$ under descending $<\kappa$-sequences.

\item Suppose that $D_\lambda \subseteq \lambda$ codes  $\{ A_{i,j}: i < \kappa, j < \kappa \}$, where for each $i < \kappa$, $\bigcup_{j<\kappa}A_{i,j}$ is a maximal antichain in $\T_{<\lambda}$ and $j_0 \neq j_1 \Rightarrow A_{i,j_0} \cap A_{i,j_1}= \emptyset$. For every $T \in \T_{<\lambda}$ and $\gamma < \kappa$, build a $\kappa$-fusion sequence $\{ T^i: i < \kappa \}$ of trees in $\clubsacks$ such that

\begin{enumerate}

\item $T=:T^0 \geq_\gamma T^1 \geq_{\gamma+1} T^2 \geq_{\gamma+2} \dots \geq_{\gamma+i} T^{i+1} \geq_{\gamma+i+1} \dots$

\item $T^i_t$ belongs to $\T_{<\lambda}$ with $\rank(T^i_t)$ at least $\lambda_i$ for $t$ in $\splitting_{\gamma+i}(T^i)$. 

\item for every $i < \kappa$, $T^{i+1} := \bigcup\{ S^{i+1}_t: t \in \splitting_{\gamma+i}(T^i) \}$, where each $S^{i+1}_t \leq T^i_t$ and $S^{i+1}_t$ hits $\bigcup_{j < \kappa}A_{i,j}$.

\end{enumerate}

Then add $T^*:=\bigcap_{i <\kappa} T^i$ to $\T_\lambda$. Moreover, for every $t \in T^*$, add $T^*_t$ to $\T^*_\lambda$  too. Finally close $\T_\lambda$ under descending $<\kappa$-sequences.

\item If $D_\lambda$ neither codes a maximal antichain (case (a)) nor an instance of $\kappa^\kappa$-bounding (case (b)), then proceed as in case (a) without its item iii. 

\end{enumerate} 

\end{enumerate}

Finally let $\T := \bigcup_{\lambda < \kappa^+}\T_\lambda$.

\begin{proposition}

The construction of $\T$ satisfies the five properties P1-P5.

\end{proposition}

\begin{proof}

P1 is clear, since at any stage we only add $\kappa$ many new trees which are in $\clubsacks$. Also P3 and P4 follow immediately from the construction.

For P2, note that the successor case $\lambda+1$ follows easily from
the construction; for $\text{cf}(\lambda) < \kappa$, start with $T \in
\T_{<\lambda}$ and use induction to build a descending
$\leq_\gamma$-descending sequence $\{ T^i:i<\text{cf}(\lambda) \}$
such that $\{ \rank(T^i): i < \text{cf}(\lambda)\}$ is cofinal in
$\lambda$ and put $T^*:= \bigcap_{i < \text{cf}(\lambda)} T^i$.  We
may additionally require that $T^i$ contains no subtree in
$\T_{<\rank(T^i)}$ and so $T^*$ contains no subtree in $\T_{<\lambda}$
(in particular, $T^*$ does not belong to $\T_{<\lambda}$); finally for
the case $\text{cf}(\lambda)=\kappa$ we argue similarly, using the
fact that we take fusion sequences with tree-ranks cofinal in
$\lambda$.

For P5, we distinguish again the three different situations. In the
successor case $\lambda+1$, we can have: case 1) $\alpha < \lambda$, so simply pick some $T_0 \supseteq T$ in $\T_\lambda$ and use the inductive hypothesis; case 2) $\alpha=\lambda$, so pick $T_0$ as in case 1) and use it as the $S$ needed to satisfy P5.
In case $\text{cf}(\lambda) < \kappa$, as
above start with $T \in \T_{<\lambda}$ and use induction to build a
descending $\gamma$-sequence of length $\text{cf}(\lambda)$ such that
$\{ \rank(T^i): i < \text{cf}(\lambda)\}$ is cofinal in $\lambda$ and
put $T^*:= \bigcap_{i < \text{cf}(\lambda)} T^i$. Let $\alpha <
\lambda$ and for $i< \text{cf}(\lambda)$ such that $\alpha <
\rank(T^i)$ choose $\gamma_i$ sufficiently large so that for all $t
\in \splitting_{\geq \gamma_i}(T^i)$ one has $T^i_t$ is contained in a
tree of $\T_\alpha \setminus \T_{<\alpha}$; then, if $\gamma^*:=
\sup\{ \gamma_i: i < \text{cf}(\lambda) \}$, for every $t \in
\splitting_{\geq \gamma^*}(T)$ one has  $T_t$ is contained in a tree
in $\T_\alpha \setminus \T_{<\alpha}$. The case
$\text{cf}(\lambda)=\kappa$ is treated similarly, by using a sequence
with tree-ranks cofinal in $\lambda$.    
\end{proof}

\begin{proposition} \label{prop:main}

$\T$ is $<\kappa$-closed, $\kappa^+$-cc and $\kappa^\kappa$-bounding.

\end{proposition}

\begin{proof}

The $<\kappa$-closure follows from point 3 of the construction.

To prove $\kappa^+$-cc we argue as follows. Let $A \subseteq \T$ be a maximal antichain and pick $\lambda$ such that $\text{cf}(\lambda)=\kappa$ and $A \cap \T_{<\lambda}$ is a maximal antichain of $\T_{<\lambda}$ and it is coded by $D_\lambda$, using $\Diamond_{\kappa^+}(S^\kappa_{\kappa^+})$. By 4.(a) of the construction, for every $T \in \T_\lambda \setminus \T_{<\lambda}$, there is $\gamma'$ such that for every $\gamma \geq  \gamma'$ for every $t \in \splitting_\gamma(T)$, $T_t$ is a subtree of some element of $A \cap \T_{<\lambda}$. By P5, if $T \in \T \setminus \T_\lambda$, there is $\gamma''\geq \gamma'$ such that for every $\gamma \geq  \gamma''$ for every $t \in \splitting_\gamma(T)$, $T_t$ is a subtree of some element of $\T_\lambda \setminus \T_{<\lambda}$. It follows that for any $T\in\T_\lambda \setminus \T_{<\lambda}$ there is $t \in T$ such that $T_t$ is a subtree of some element of $A \cap \T_{<\lambda}$, and therefore $A \cap \T_{<\lambda}$ is a maximal antichain in $\T$. So $A \cap \T_{<\lambda}=A$, which finishes the proof as $|\T_{<\lambda}|=\kappa$.











For $\kappa^\kappa$-bounding we argue as follows. Let $\dot x$ be an $\T$-name for an element of $\kappa^\kappa$ and $T \in \T$. Choose $\{A_{i,j}: i < \kappa, j < \kappa  \}$ so that for each $i < \kappa$, $\bigcup_{j < \kappa} A_{i,j}$ is a maximal antichain and elements of $A_{i,j}$ force $\dot x(i) = j$. Pick $\lambda < \kappa$ such that $T$ belongs to $\T_{<\lambda}$, $\text{cf}(\lambda)=\kappa$ and $D_\lambda$ codes such $\langle A_{i,j} \cap \T_{<\lambda}: i,j \in \kappa \rangle$. and they are maximal antichains as in 4(b). Hence, we can build a $\kappa$-fusion sequence in order to get $T' \leq T$ such that for each $i<\kappa$, $T'$ forces the generic to hit $\bigcup_{j \in J_i} A_{ij}$, where each $J_i \subseteq \kappa$ has size $\leq 2^i$. Define $z \in \kappa^\kappa \cap V$ by $z(i)=\sup J_i$; then $T' \force \forall i < \kappa, \dot x (i) \leq z(i)$.

\end{proof}

\section{Ideal and measurability} \label{ideal}

Once we have a tree forcing notion we can introduce a related ideal of \emph{small} sets.

\begin{definition}

A set $X \subseteq 2^\kappa$ is said to be $\T$-null iff for all $T
\in \T$ there exists $T' \in \T$, $T' \leq T$ such that $[T'] \cap X =
\emptyset$. Further let $\IF$ be the ideal consisting of all $\T$-null
sets. A set is $\T$-conull if its complement is in $\IF$.

\end{definition}

\begin{remark}

$\IF$ is a $\kappa^+$-ideal; let $\{ X_\alpha: \alpha < \kappa \}$ be a sequence of $\T$-null sets, and fix $T \in \T$. Using 4(b) of the construction of $\T$, build a $\kappa$-fusion sequence $\{ T_\alpha: \alpha < \kappa \}$ such that for all $\alpha < \kappa$, for all $\beta \leq \alpha$, $[X_\beta] \cap [T_\alpha]= \emptyset$. Then $T':= \bigcap_{\alpha < \kappa} T_\alpha$ has the desired property.

\end{remark}

One of the main properties of the null ideal in the standard framework is that of being \emph{orthogonal} to the meager ideal, i.e., the space can be partitioned into a meager piece and a null piece. We now prove that the same holds for $\IF$.

\begin{proposition}

There is $X \subseteq 2^\kappa$ such that $X \in \meager$ and $2^\kappa \setminus X \in \IF$.

\end{proposition}

\begin{proof}

Let $A:= \{ A_i: i < \kappa \}$ be a maximal antichain in
$\T$. Clearly, $X:= \bigcup_{i < \kappa} [A_i]$ is $\T$-conull, since
for every $T \in \T$, there is $i < \kappa$ such that $A_i$ and 
$T$ are compatible, and so there is $T' \leq A_i$ such that $T' \leq T$. It is then
sufficient to show that we can find such an antichain $A$ with the
further property that any $[A_i]$ is nowhere dense. But note that by
property P2, any $T\in\T$ can be extended to contain no subtree of the
form $(2^{<\kappa})_s$ for $s\in 2^{<\kappa}$ and $[T]$ is nowhere
dense for such a tree $T$. 
Now let $\T^* \subseteq \T$ be the dense set of such trees, and pick
$A$ a maximal antichain in $\T^*$. Then $A$ remains a maximal
antichain in $\T$ as well, and it is then enough for our purpose.  

\end{proof}

\paragraph{Measurability. }

There are essentially two possible notions of regularity related to $\T$.

\begin{definition} \label{def:measurable}
A set $X \subseteq 2^\kappa$ is said to be:
\begin{enumerate}
\item \emph{$\T$-measurable} iff for every $T \in \T$ there exists $T' \in \T$, $T' \leq T$ such that $[T'] \setminus X \in \IF$ or $X \cap [T'] \in \IF$.
\item \emph{$\T$-regular} iff there exists a Borel set $B$ such that $X {\vartriangle} B \in \IF$.
\end{enumerate}
\end{definition}
Concerning definition \ref{def:measurable}.1, we could equivalently require ``$[T'] \subseteq X$ or $[T'] \cap X = \emptyset$'', as $\IF$ is a $\kappa^+$-ideal.

\begin{proposition}  \label{meas-regular}
Let $X \subseteq 2^\kappa$.
$X$ is $\T$-measurable iff $X$ is $\T$-regular.
\end{proposition} 

\begin{proof}

The proof is just as the general case of $\poset{P}$-measurability in the standard case. We give it here for completeness.

$\Rightarrow$: by assumption, the set $E:=\{ T \in \T: [T] \cap X \in \ideal{I}_\T \vee [T] \cap X^c \in \ideal{I}_\T \}$ is dense in $\T$. Then pick a maximal antichain $A$ in $E$ and put 
\[
B:= \bigcup\{ [T]: T \in A \land [T]\cap X^c \in \ideal{I}_\T \}.
\]

Note that $B$ is Borel ($\SSigma^0_2$), since $|A| \leq \kappa$. We claim that $X {\vartriangle} B \in \ideal{I}_\T$. Indeed, for every $T \in \T$ we have two cases: there is $S \leq T$ such that $[S] \cap X \in \ideal{I}_\T$; if we pick the unique $S' \in A$ such that $S \leq S'$ we get $[S] \cap (X {\vartriangle} B ) \in \ideal{I}_\T$, as $S' \in A$ implies $[S] \cap X^c \in \ideal{I}_\T$; otherwise, the specular situation occurs, i.e., there is $S \leq T$ such that $[S] \cap X^c \in \ideal{I}_\T$; so if we pick the unique $S' \in A$ such that $S \leq S'$ we get $[S] \cap (X {\vartriangle} B) \in \ideal{I}_\T$, as $S' \in A$ implies $[S] \subseteq B$.

$\Leftarrow$: by assumption, it is enough to show that any Borel set is $\T$-measurable. We do that by induction on the Borel hierarchy. By point 2 in the main construction of $\T$, we can generalize a result of Brendle and L\"owe (see \cite{BL99}), proving that all Borel sets are $\T$-measurable is equivalent to proving that for every Borel set $B$ there is $T \in \T$ such that $[T] \subseteq B$ or $[T] \cap B = \emptyset$. For $s \in 2^{<\kappa}$, $[s]$ is trivially $\T$-measurable, and if $B$ is $\T$-measurable, then by symmetry $B^c$ is $\T$-measurable too. Finally, if $B$ is the union of $\leq \kappa$ many Borel sets $\{ C_\alpha: \alpha < \delta \}$, with $\delta \leq \kappa$, then we have two cases: there is $\alpha < \delta$ and $T \in \T$ such that $[T]\subseteq C_\alpha \subseteq B$; or for all $\alpha<\delta$, $C_\alpha \in \ideal{I}_\T$, and so $B \in \ideal{I}_\T$ as well.    

\end{proof}

In \cite{L14} and \cite{S13} it was shown that for some tree forcing notions one can force all projective sets to be measurable. Nevertheless this is not the case for $\T$.
Indeed, next result shows that there is no hope for $\SSigma^1_1(\T)$ to be consistent.
\begin{proposition}
The club filter $\club$ is not $\T$-measurable.
\end{proposition}

\begin{proof}
Standard. Given an arbitrary $T \in \T$ it suffices to show that $[T]$ contains branches both in $\club$ and in $\ns$. We argue as follows: 
\begin{itemize}
\item let $t_0= \stem(T)$
\item for $\alpha < \kappa$ successor, pick $t_\alpha \supset {t_{\alpha-1}}^\conc 1$ such that $t_\alpha \in \splitting(T)$ and $|t_\alpha|$ is a limit ordinal.
\item for $\alpha < \kappa$ limit, pick $t_\alpha \supset {(\bigcup_{\xi<\alpha} t_\xi)}^\conc 1$ such that $t_\alpha \in \splitting(T)$ and $|t_\alpha|$ is a limit ordinal.
\end{itemize}

Finally put $x:= \bigcup_{\alpha<\lambda} t_\alpha$. Clearly, $x \in [T] \cap \club$.

Analogously, if in the previous choices of the $t_\alpha$'s we replace 1 by 0, we get $x \in [T] \cap \ns$. 

\end{proof}

We conclude by remarking that a standard construction shows that $\DDelta^1_1(\T)$ is consistently false (e.g., if $V=L$).

\section{Concluding remarks and open questions}
It would be interesting to prove that $\DDelta^1_1(\T)$ is consistently true. The usual way for forcing such a statement for a tree-forcing $\poset{P}$ is to take a $\kappa^+$-iteration of $\poset{P}$ with $\kappa$-support. The main point is to make sure that $\kappa^+$ is preserved. For our forcing $\T$ it is not clear whether it is the case. So we leave the following as an open question.

\begin{itemize}
\item[\textbf{Question}.] Does a $\kappa^+$-iteration of $\T$ with $\kappa$-support preserve $\kappa^+$? Or, can we modify $\T$ is order to let the latter work?
\end{itemize}

What remains also open is the last part of \cite[Question 3.1]{KLLS16}.

\begin{itemize}
\item[\textbf{Question}.] Can one define a tree forcing that is $<\kappa$-closed, $\kappa^\kappa$-bounding and $\kappa^+$-cc, for $\kappa$ successor?
\end{itemize}

We remark that our construction can in fact be applied to successor $\kappa$, yielding $< \kappa$-closure and the $\kappa^+$-cc, but $\kappa^\kappa$-bounding will fail.

\end{document}